\documentclass[a4paper,10pt,reqno]{article}
\usepackage[colorlinks=true]{hyperref}
\usepackage{amssymb}
\usepackage{mathtools}
\usepackage{comment}
\usepackage[makeroom]{cancel}
\usepackage{tikz-cd}
\usepackage[linguistics]{forest}
\usepackage{tikz-qtree}
\forestset{sn edges/.style={for tree={parent anchor= north, child anchor=south}}}
\forestset{
fairly nice empty nodes/.style={
delay={where content={}
{shape=coordinate, }{}},
for tree={s sep=4mm}
}
}
\forestset{
nice empty nodes/.style={
    for tree={calign=fixed edge angles},
    delay={where content={}{shape=coordinate, for children={anchor=south} }{}}
},
}

\usepackage{adjustbox}

\usepackage[utf8]{inputenc}

\usepackage{latexsym,amsfonts,amsthm,amsmath,amscd,amssymb,color}
\usepackage[all]{xy}

\theoremstyle{plain}
\newtheorem{lemma}{Lemma}[section]

\newtheorem{proposition}[lemma]{Proposition}

\newtheorem{question}[lemma]{Question}

\theoremstyle{definition}
\newtheorem{definition}[lemma]{Definition}

\newtheorem{example}[lemma]{Example}
\newtheorem{remark}[lemma]{Remark}
\newtheorem*{definition*}{Definition}
\theoremstyle{remark}

\newcommand{\lb}{\left[ \cdot\,,\cdot\right] }
\usepackage[left=2cm, right=3cm, top=3cm, bottom=1.7cm]{geometry}

\title{On longitudinal differential operators and Nash blowups}

\usepackage{authblk}

\author{Ruben Louis\thanks{Department of Mathematics, University of Illinois Urbana-Champaign, Urbana, IL, USA.\\Email: \texttt{rlouis@illinois.edu}}}
\date{\today}

\begin{document}

\maketitle

\begin{abstract}
    In this short note, we describe the Helffer–Nourrigat cone \cite{Helffer1980, androulidakis2022pseudodifferential} of a singular foliation in terms of the Nash algebroid associated to the foliation \cite{mohsen2021blow, louis2023series}. Along the way, we show that the Helffer–Nourrigat cone is a union of symplectic leaves of the canonical Poisson structures on the dual of the holonomy Lie algebroids. We also provide, within this framework, a characterization of longitudinally elliptic differential operators on a singular foliation $\mathcal{F}$, generalizing results previously known in the literature.
\end{abstract}

\section*{Introduction}
A differential operator of rank $k\in \mathbb N$ on a manifold $M$ is a local endomorphism $D$ of the algebra of smooth functions $\mathcal{O}_M$ of $M$, which, in local coordinates, can be expressed as a  $\mathcal{O}_M$-linear combination  of compositions of vector fields on $M$, i.e.,
\[{\displaystyle D=\sum _{|\ell |\leq k}f_{\ell }(x)\frac {\partial ^{|\ell|}}{\partial x_{1}^{\ell _{1}}\partial x_{2}^{\ell_{2}}\cdots \partial x_{q}^{\ell_{q}}}\,}\] where for $q\in \mathbb N$,  $\ell = (\ell_1 , \ell_2 , \cdots, \ell_q)$ is a multi-index of natural numbers and $|\ell| = \ell_1 + \ell_2 +\cdots + \ell_q $. To every differential operator $D$ of degree $k$, and any choice of local coordinates, one associates its symbol $\sigma_D\in \Gamma(S^{\leq k}(TM))$ which is a function on $T^*M$ that is given in local chart by the expression: \[{\displaystyle \sigma_D(m,v )=\sum _{|\ell |\leq k}f_{\ell}(x) v_1^{\ell_{1}}\odot\cdots\odot  v_q^{\ell_{q}}}.\] The component of the highest degree of the symbol $\sigma_D(m,v)$, that is, \( \sigma (m,v)=\sum _{|\ell |=k}f_{\alpha}(x)v_1^{\ell_{1}}\odot\cdots\odot  v_k^{\ell_{k}}\), does not depend on the local coordinate charts when seen as a homogeneous polynomial function of degree $k$ on $T^*M$; it is called the \emph{principal symbol} of $D$ \cite{Hormander1998}. {Differential operators of order $k$ are also defined on Lie groupoids as left-invariant differential operators, that is, differential operators valued in Lie algebroids} $A\to M$ \cite{NistorPingAlan,bekaert2023geometric}, in particular, we look at differential operators on a foliated manifold, that is, a manifold $M$ together with a regular foliation $\mathcal{F}$, which we see as a Lie algebroid $A\subseteq TM$ with injective fiber-wise anchor map. These differential operators are called longitudinal differential operators, as they are compatible with the leaf decomposition of $(M, \mathcal{F})$. The symbol in this case is a section of $S^{\leq k}(A)$, that is, a function on $A^*$ \cite{Connes1982ASO,NistorPingAlan, zbMATH07369649}.

In this paper, we are interested in longitudinal differential operators on a singular foliation, that is, on a submodule of vector fields of $\mathfrak X(M)$ that is closed under the Lie bracket and locally finitely generated. More precisely, as in \cite{LLL1}, we unify \cite{Hermann,Cerveau,AndroulidakisIakovos, Debord,LLS} in smooth differential geometry and \cite{zbMATH03423310,Ali-Sinan} in holomorphic differential geometry by defining a singular foliation on a smooth, complex, algebraic, or real analytic manifold $M$, with sheaf of functions $\mathcal O_M$, to be a subsheaf $\mathcal F \colon U\longrightarrow\mathcal F(U )$ of the sheaf of vector fields $\mathfrak X$, which is closed under the Lie bracket and locally finitely generated as an $\mathcal O_M$-module. 
By Hermann's theorem \cite{Hermann, Sussmann1, Sussmann2,Stepan1, Stepan2}, this is enough to induce a partition of the manifold $M$ into immersed submanifolds of possibly different dimensions, called \emph{leaves} of the singular foliation. Singular foliations generalize the notion of regular foliations by allowing leaves with possibly different dimensions. They appear, for instance, as orbits of Lie group actions, or, more generally, as the orbits of a Lie groupoid or the image of the anchor map of a Lie algebroid. 

For a regular foliation $F\subseteq TM$, longitudinal differential operators coincide with the universal enveloping algebra of the Lie algebroid $F$ {in the sense of \cite[\S 3]{NistorPingAlan} or \cite[Definition 1.3]{Moerdijk2010}}. More generally, for a Debord foliation $\mathcal{F}\subseteq \mathfrak X(M)$, i.e., $\mathcal{F}$ is isomorphic to the module of sections of a Lie algebroid whose anchor map is injective on an open dense subset, this is still true. This is \underline{not} the case for a generic singular foliation. 
We investigate the following question : 
\begin{question}\label{thm:symbol}
What is the symbol of a longitudinal differential operator on a singular foliation $\mathcal F$?
\end{question}
Androulidakis and Skandalis in \cite{zbMATH05843246} gave an answer that is too involved to be dealt with here, using the $C^*$-algebra of half densities of the holonomy Lie groupoid.
Later on, Mohsen, then Androulidakis, Mohsen and Yuncken's \cite{mohsen2021blow}-\cite{androulidakis2022pseudodifferential} gave a second elaborate answer involving representations, more involved than the one considered in the present paper. See also \cite{debord2025hypoellipticit}. In the process, they showed that Mohsen's blowup \cite{mohsen2021blow} provides an answer to Question \ref{thm:symbol} sufficient for several purposes, including dealing with pseudo-differential operator calculus developed by Androulidakis and Skandalis in \cite{zbMATH05843246}, using the holonomy Lie groupoid and bisubmersions. We present an interpretation of this answer. {In \cite{louis2024nash,louis2023series}, an interpretation of Mohsen's blow-up is provided in terms of an anchored bundle \( A \to TM \) over a singular foliation \( \mathcal{F} \), that is, a bundle \( A \) satisfying \( \rho(\Gamma(A)) = \mathcal{F} \). It is shown that the associated blow-up map \( p : M_{\mathcal{F}} \to M \) has the property that the pullback foliation \( p^{-1}(\mathcal{F}) \) becomes a Debord foliation on \( M_{\mathcal{F}} \).} In other words, $p^{-1}(\mathcal{F})$ is the image of the anchor map of a Lie algebroid over $M_\mathcal{F}$, denoted \( \mathfrak{D}_{\mathcal{F}}\), whose anchor is injective on a dense open subset of \( M_{\mathcal{F}} \). We refer to Mohsen's blowup as the \emph{Nash blow-up} of \( \mathcal{F} \), because it clearly belongs to that family and matches the blowups of \cite{Rossi,Ali-Sinan, JPTS} for coherent sheaves. The latter is the Lie algebroid of Mohsen's groupoid \cite{mohsen2021blow}.

This construction allows us to define the symbol of a longitudinal operator on \( \mathcal{F} \) either as a function on the dual of the Lie algebroid \( \mathfrak{D}_{\mathcal{F}} \) or as a function of the Helffer-Nourrigat cone. The first one being the pull-back of the second one. This symbol is well defined and is related to the leafwise symbol associated to the holonomy Lie algebroids of the leaves of \( \mathcal{F} \) via the so-called \emph{Helffer--Nourrigat cone} \cite{Helffer1980,androulidakis2022pseudodifferential}, which we describe in these terms in the text \S \ref{sec:Helffer}. In this framework, we prove that if \( \mathcal{F} \) is the image of a Lie algebroid \( A \), then the Helffer--Nourrigat cone is a union of symplectic leaves of the linear Poisson structure on \( A^* \).

{The symbol considered in the present paper should be viewed as a geometric incarnation of the principal symbol appearing in the pseudodifferential calculus of Androulidakis-Mohsen-Yuncken \cite{androulidakis2022pseudodifferential}. In \cite{androulidakis2022pseudodifferential}, the principal symbol of a longitudinal operator is defined through representations of the $C^*$-algebra of the holonomy groupoid and is naturally evaluated on the Helffer–Nourrigat cone. In our framework, the Nash algebroid $\mathfrak D_\mathcal F \to M_\mathcal F$ provides a resolution of this cone: Proposition \ref{prop:NHandNash} identifies the dual bundle $\mathfrak D_\mathcal F^*$ with a desingularization of the Helffer-Nourrigat cone, and Proposition \ref{prop:pullbackHN} shows that the symbol obtained through the Nash algebroid is exactly the pullback of the leafwise symbol restricted to that cone. Thus, while our definition does not reconstruct the full analytic machinery of \cite{androulidakis2022pseudodifferential}, it provides a geometric model for the support on which the Androulidakis-Mohsen-Yuncken principal symbol lives.
}

The paper is structured as follows: In \ref{sec:SF}, we recall the notions of singular foliations, the Nash blow-up of a singular foliation, and define the Helffer–Nourrigat cone of a singular foliation in terms of anchored bundles, relating it to the Nash blow-up.  In \ref{sec:LDO}, we discuss various definitions of the symbol of a longitudinal operator on a singular foliation and study the relationships between them.


\section{Singular foliations}\label{sec:SF}
Throughout the paper, $M$ stands for a smooth/analytic/complex manifold or an algebraic variety over $\mathbb K\in \{\mathbb R, \mathbb C\}$. We denote by $\mathcal{O}_M$ the algebra of smooth/holomorphic or polynomial functions on $M$ depending on the context. In addition, $\mathbb K$ stands for $\mathbb R$ or $\mathbb C$.\\

Here, following \cite{Hermann,Cerveau,Debord,LLS} and \cite{zbMATH03423310,Ali-Sinan} or \cite{LLL1}, we define
\begin{definition}\label{def:singfoliation}
    A \emph{singular foliation} on a smooth, real analytic, or complex manifold $M$   is a subsheaf $\mathcal{F}\subseteq\mathfrak{X}(M)$ of modules over functions that fulfills the following conditions
\begin{enumerate}
     \item \textbf{Stability under Lie bracket}: $[\mathcal{F},\mathcal{F}]\subseteq \mathcal{F}$.
     \item \textbf{Local finite generateness}: every $m\in M$ admits an
open neighborhood $\mathcal{U}\subseteq M$ together with a finite number of vector fields $X_1, \ldots, X_r \in \mathcal F$ such that for
every open subset $\mathcal V \subseteq \mathcal U$ the vector fields ${X_1}_{|_\mathcal{V}} ,\ldots, {X_r}_{|_\mathcal{V}}$ generate $\mathcal{F}$ on $\mathcal{V}$ as a $\mathcal{O}_\mathcal{V}$-module.
 \end{enumerate}
 \end{definition}One of the important consequences of the above definition in the smooth/real analytic/complex cases (see \cite{Hermann} also \cite[\S 1.7]{LLL1}) is that singular foliation admits leaves : There exists a partition of $M$ by immersed submanifolds $L\subseteq M$ called \emph{leaves} such that for all $m\in M$, the image of the evaluation map $\mathcal F \to T_m M  $  is the tangent space $T_mL$ of the leaf $L$ containing $m$. We denote by $M/\mathcal{F}$ the space of leaves of $\mathcal{F}$.
     
 \begin{example} \label{ex:SF}Here are some important classes of singular foliations.
    \begin{enumerate}
        \item One very important source of examples of singular foliations  is Poisson geometry. Every Poisson manifold $(M,\pi)$ comes equipped with a singular foliation $\mathcal{F}_\pi$ on $M$.  Recall that a Poisson structure on a manifold $M$ is a bivector field $\pi=\{\cdot\,,\cdot\}\in \mathfrak X^2(M)$ fulfilling $[\pi, \pi]_{\mathrm{NS}}=0$ where $[\cdot\,,\cdot]_{\mathrm{NS}}$ is the Schouten–Nijenhuis bracket on multivector fields. For every function $h\in \mathcal O_M$ we associate a vector field $X_h:=\pi^\#(dh)=\{\cdot\,,h\}\in\mathfrak X(M)$ called the \emph{Hamiltonian} vector field of $h$. One has that $[X_h, X_g]=X_{\{g,h\}}$ for all $g,h\in \mathcal O_M$. This implies that the $\mathcal{O}_M$-module $\mathcal{F}_\pi\subset \mathfrak X(M)$ generated by the Hamiltonian vector fields is a singular foliation on $M$. It is well-known that the leaves of $\mathcal{F}_\pi$ are symplectic manifolds \cite{LPV, CFM, LLL1}. Another related class of examples of singular foliations are provided by twisted-Poisson structures, i.e.,
$[\pi ,\pi ]_{\mathrm{NS}}=\wedge ^{3}\pi ^{\#}(\alpha )$, where
$\alpha \in \Omega ^{3}(M)$ is a closed 3-form \cite{SW}. More generally, one may also consider foliated bivector vector fields \cite{Turki}. 

\item Lie algebroids form a large class of examples of singular foliations. Recall that the Lie algebroid is a triple $(A, [\cdot\,,\cdot]_A,\rho)$ where $(\Gamma(A),[\,\cdot,\cdot]_A)$ is a Lie algebra and $\rho\colon A\to TM$ is a vector bundle morphism (referred to as \emph{the anchor map}) that satisfies the Leibniz rule, i.e., for all $a,b\in \Gamma(A), f\in \mathcal{O}_M$, $[a,fb]_A=\rho(a)[f]b+f[a,b]_A$. The image $\rho(\Gamma(A))\subset \mathfrak X(M)$ of the anchor map generates a singular foliation on $M$. One should notice that the singular foliation that is given by the symplectic leaves of a Poisson manifold $(M, \pi)$ is a subclass of singular foliations arising from Lie algebroids as  $\mathcal{F}_\pi$ is generated by the image of the cotangent Lie algebroid $\pi^\#\colon T^*M\to TM$. See e.g., \cite[\S 2.2]{CFM}. 
\item Another important class of singular foliations is the so-called Debord foliations. A singular foliation $\mathcal F$ is \emph{Debord} if it is projective as a module over functions on $M$, equivalently  if and only if there exists a Lie algebroid $(A, \lb_A, \rho)$ such that $\rho(\Gamma(A))=\mathcal F$ whose anchor is injective on an open dense subset. In particular, Debord foliations are globally finitely generated projective submodules of $\mathfrak X(M)$. Conversely, given any finitely generated singular foliation $\mathcal{F} \subseteq \mathfrak{X}(M)$ that is projective as a $\mathcal{O}_M$-module, there exists a vector bundle $A$ such that $\Gamma(A) \simeq \mathcal{F}$, by the Serre–Swan Theorem \cite{SwanRichardG, MoryeArchanaS}. This vector bundle carries a natural Lie algebroid structure whose anchor map $A \to TM$ is injective on an open dense subset.  One of the key properties of Debord foliations is that they are always integrable to smooth Lie groupoids, in contrast to arbitrary Lie algebroids, which may not be integrable \cite{Debord, Crainic-Fernandes}. Debord Lie algebroids are naturally the class of Lie algebroids closest to regular foliations, that is, to vector subbundles $F\subseteq TM$ such that $\left [\Gamma(F), \Gamma(F)\right]\subset \Gamma(F)$. 
\item There are examples of singular foliations that do not arise from natural Lie algebroids even locally (in \cite[Lemma 1.3]{AZ}). We even have an embarrassing example where the question of the existence of a Lie algebroid is still open: On $M=\mathbb R^2$ take $\mathcal{F}\subset \mathfrak X(\mathbb R^2)$ to be the singular foliation generated by the vector fields that vanish at order two, namely $\mathcal F$ is globally generated by the vector fields \[x^2\frac{\partial }{\partial x},\, y^2\frac{\partial }{\partial x},\, xy\frac{\partial }{\partial x},\, x^2\frac{\partial }{\partial y},\, y^2\frac{\partial }{\partial y},\, xy\frac{\partial }{\partial y}.\]
It is not known whether $\mathcal F$ is the image of a Lie algebroid. However, it is the image of the anchor map of a universal Lie $2$-algebroid. See \cite{LLS}.
    \end{enumerate}
 \end{example}

\subsection{Anchored vector bundles over a singular foliation}\label{sec:kernel}
We mentioned in \S \ref{sec:SF}, Example \ref{ex:SF} (4) that not all singular foliations arise from Lie algebroids. However, if $\mathcal F$ is globally finitely generated, then it is the image of an anchored bundle. We refer the reader to \cite[\S 2.1.1]{LLL2} for a complete discussion of this matter. 
\begin{definition}
    Let $(M, \mathcal{F})$ be a singular foliation. A vector bundle morphism $\rho\colon A\to TM$ is said to be an \emph{anchored bundle over $\mathcal F$} if $\rho(\Gamma(A))=\mathcal F$.
\end{definition}


Let \(A\to TM\) be an anchored bundle over a singular foliation $\mathcal{F}$. There is an almost-Lie algebroid structure\footnote{That is, the triple $(A, [\,\cdot,\cdot]_A, \rho)$ satisfies all the axioms of a Lie algebroid except Jacobi identity. See e.g., \cite{popescu2019almost}.} $[\,\cdot,\cdot]_A$ on $A$ (which always exists in the smooth case, and locally in other cases; we owed this to Marco Zambon). Recall \cite[\S 2.3.1]{LLL2} that the {strong-kernel} $\mathrm{Sker}\,\rho_m\subseteq A_m$ of $\rho$ at $m\in M$ is the subvector space of $\ker\rho_m\subseteq A_m$ defined as $u\in A_m$ such that there is an open neighborhood $\mathcal{U}$ of $m$ and a local section $\Tilde{u}\in\Gamma(A)|_{\mathcal{U}}$ through $u$ so that $\rho(\Tilde{u})=0$. This allows us to define several objects that are canonically attached to $\mathcal F$\begin{enumerate}
    \item The quotient
 $$ \mathfrak g_m (\mathcal F) =
 \frac{\ker\rho_{m}}{\mathrm{Sker}\,\rho_m}$$ is a Lie algebra that is called the
 \emph{isotropy of $\mathcal F$ at $m$}. The latter is isomorphic to the usual definition of isotropy Lie algebra $\frac{\mathcal{F}(m)}{\mathcal{I}_m\mathcal{F}}$ of a singular foliation of Androulidakis-Skandalis \cite{AndroulidakisIakovos}. Here $\mathcal{\mathcal{F}}(m)=\{X\in \mathcal{F}\;|\; X(m)=0\}$ and $\mathcal{I}_m=\{f\in C^\infty(M)\;|\; f(m)=0\}$.
    \item The restriction $\mathrm{Sker}\rho|_{L}$ of the strong-kernel $\mathrm{Sker}\rho$ to a leaf $L\in M/\mathcal{F}$ is a vector bundle  over $L$, and the quotient vector bundle \[A_L:=\frac{A|_L}{\mathrm{Sker}\rho|_L}\to L\] comes equipped with a natural Lie algebroid structure over $L$. For every point $m\in M$ and any open neighborhood $\mathcal{U}\subseteq M$ of $m$ such that the connected component $(L\cap \mathcal{U})_m$ of $m$ in $L\cap \mathcal{U}\hookrightarrow \mathcal{U}$ is embedded in $\mathcal{\mathcal U}$ there exists an isomorphism \begin{equation}
        \label{eq:holonomyalg}\Gamma(A|_{(L\cap \mathcal{U})_m})\simeq \frac{\mathcal{F}|_\mathcal{U}}{\mathcal{I}_{(L\cap \mathcal{U})_m}\mathcal{F}|_\mathcal
 U}.\end{equation} Here,  $\mathcal{I}_{(L\cap \mathcal{U})_m}$ the ideal of functions vanishing
on $\mathcal{I}_{(L\cap \mathcal{U})_m}$. Since the right-hand side of Equation \eqref{eq:holonomyalg} does not make any reference to the
anchored bundle $A\to TM$, it means that $A_L$ does not depend on the choice of a particular anchored bundle. Notice that the anchored bundle $A$ does not need to be defined on the whole manifold of $M$, just on open neighborhoods of points in $L$. See \cite[\S2.3.6]{LLL2}.
\end{enumerate}

Although a singular foliation $(M, \mathcal{F})$ may not be the image of a Lie algebroid over $M$, it always admits a transitive Lie algebroid longitudinally, that is, along the leaves of $\mathcal{F}$. 

\begin{proposition}\cite{AndroulidakisIakovos, AZ}
    Let $\mathcal{F}$ be a singular foliation over $M$ and $L$ a leaf of $\mathcal F$. There exists a transitive Lie algebroid (referred to as the \emph{holonomy Lie algebroid of $\mathcal{F}$ on $L$}) $A_L\to L$ whose image of anchor map $\rho_L\colon  A_L \to TL$ is $\mathcal{F}|_L$. Moreover, 
    \begin{enumerate}
        \item if $L\subset M$ is embedded, then the $C^\infty(L)$-module $\Gamma(A_L)$ of sections of $A_L$  is isomorphic to $\frac{\mathcal{F}}{\mathcal{I}_L\mathcal{F}}$, where $\mathcal{I}_L$ is the ideal of functions vanishing on $L$.
        \item there is a short exact sequence:
 $$  \xymatrix{ \mathfrak g_L(\mathcal F)\ar@{^(->}[r]& A_L \ar^{\rho_L}@{->>}[r]& TL}$$
where $\mathfrak{g}_L(\mathcal F)=\bigsqcup_{m\in L}\mathfrak{g}_m(\mathcal F)$ is a bundle of Lie algebras over $L$ whose fiber at $m \in L$ is the isotropy Lie algebra $\mathfrak g_m (\mathcal F)$. 
    \end{enumerate}
\end{proposition}
\subsection{Nash-blowup of a singular foliation}\label{sec:Nash}  In the sequel, we assume that \begin{enumerate}
\item[(a)] our singular foliation $\mathcal{F}$ is finitely generated so that there is an anchored bundle $\rho\colon A\to TM$ so that $\rho(\Gamma(A))=\mathcal{F}$.
    \item[(b)] the regular leaves of  $\mathcal F$ are all of the same  dimension $r\in \mathbb N$.
\end{enumerate}Recall \cite{mohsen2021blow,louis2024nash,louis2023series} that the \emph{Nash blow-up} $\left(M_\mathcal{F}, p^{-1}(\mathcal{F})\right)$ of a singular foliation $(M,\mathcal{F})$ is such that \begin{enumerate}
    \item $p\colon M_\mathcal{F}\to M$ is proper and surjective,
    \item $p^{-1}(M_{\mathrm{reg}})$ is dense in $M_\mathcal{F}$ and the restriction of $p$ to $p^{-1}(M_{\mathrm{reg}})$ is a bijection,
    \item the pullback $p^{-1}(\mathcal{F})$ of the singular foliation $\mathcal{F}$ along $p$ exists, \begin{enumerate}
        \item and it is Debord, i.e., $p^{-1}(\mathcal{F})$ is the image through the anchor map of a Lie algebroid\footnote{The pullback $p^{-1}(\mathcal{F})$ and the Lie algebroid $D_\mathcal
    F$ make sense even if $M_\mathcal{F}$ is not a smooth manifold \cite{louis2024nash,louis2023series}.} $\mathfrak{D}_\mathcal{F}$ (referred here as \emph{Nash algebroid of $\mathcal{F}$}) over $M_\mathcal{F}$ whose anchor map is injective on the open dense subset $p^{-1}(M_{\mathrm{reg}})$. For every $X\in \mathcal{F}$, the pullback of $X$ along $p$ exists, and we denote it by $p^!X\in p^{-1}(\mathcal F)$, it is the unique vector field on $M_\mathcal
    F$ that is $p$-projectable to $X$.
    \item $\mathfrak D_\mathcal{F}$ sits in an exact sequence of vector bundles over $M_\mathcal{F}$\begin{equation}\label{eq:exactsequence}
        0\to T\to p^*A\to \mathfrak D_\mathcal{F}\to 0.
    \end{equation}If $A$ has a Lie algebroid structure, the sequence \eqref{eq:exactsequence} is an exact sequence of Lie algebroids and $T$ is a Lie algebra bundle \cite{louis2024nash}.
    \end{enumerate}
\end{enumerate} The \emph{Nash blow-up space} $M_\mathcal{F}$  (which depends only on $\mathcal{F}$) is the closure $ \overline{\sigma(M_{\mathrm{reg}})}\subset \mathrm{Grass}_{-r}(A)$ of the image of the section \[\sigma\colon M_{\mathrm{reg}}\longrightarrow \mathrm{Grass}_{-r}(A),\, x\longmapsto \ker\rho_x\]
where $P\colon \mathrm{Grass}_{-r}(A)\to M$ is the Grassmann bundle of vector subspace of codimension $r$. Moreover $p\colon M_\mathcal{F}\to M$ is the restriction of $P$ to $M_\mathcal{F}$.


\begin{remark}
\label{rmk:limitesGrassmann}
 \begin{enumerate}
     \item Intuitively, for $x\in M$, $p^{-1}(x)=M_\mathcal{F}\cap P^{-1}(x)$ is the set of all possible limits of the subspaces $\ker \rho_{y}$ when $y\in M_{\mathrm{reg}}$ converges to $x$.
More precisely, for any $x\in M$, there is an open neighborhood $\mathcal U\subset M$ of $x$ such that $\mathrm{Grass}_{-r}(A)\simeq \mathcal{U}\times \mathrm{Grass}_{-r}(\mathbb{K}^{\mathrm{rk}(A)})$. By construction,   \begin{equation*}
        p^{-1}(x)=\left\{V\subset A_{_x}\;\middle|\; \exists\, (x_n)\in M_{\mathrm{reg}}^\mathbb N,\, \text{such that},\,\;\ker \rho_{x_n} \underset{n \to +\infty}{\longrightarrow} V \text{ as }  x_n\underset{n \to +\infty}{\longrightarrow}x\right\}.
        \end{equation*}
        \item In general, the space \( M_\mathcal{F} \subset \mathrm{Grass}_{-r}(A) \) is not smooth. However, it makes sense to say that the pullback \( p^{-1}(\mathcal{F}) \) is a singular foliation on \( M_\mathcal{F} \), and the latter admits smooth leaves within \( M_\mathcal{F} \)~\cite{louis2023series}. In~\cite{Ali-Sinan} and~\cite[\S 2.1.2]{louis2023series}, certain criteria are provided to determine whether \( M_\mathcal{F} \) is a smooth manifold.
   \end{enumerate}
\end{remark}Some important features of the Nash blow-up $p\colon M_\mathcal{F}\to M$ are the following:\begin{enumerate}
\item {For every $x\in M$ and $V\in p^{-1}(x)$,  $\mathrm{Sker}\,\rho_x\subseteq V\subseteq \ker\rho_x$.} 
\item For every $x\in M$ and $V\in p^{-1}(x)$, the image $ \overline V$ of $V$ of the isotropy Lie algebra , 
$$ \mathfrak g_x (\mathcal F) =
 \frac{\ker\rho_{x}}{\mathrm{Sker}\,\rho_x}$$ is a Lie subalgebra 
 of codimension $r-\dim (\mathrm{Im}\,\rho_x)$.
 Also, $V \to \overline{V} $ is an injective map from $p^{-1}(x) $ to ${\mathrm{Grass}}_{-(r - \dim (\mathrm{Im}\,\rho_x))} (\mathfrak g_x (\mathcal F)  )  $.
\end{enumerate}
\begin{definition}
    For every $m\in M$, the set of Lie algebras $\overline{V}\subset \mathfrak g_m(\mathcal{F})$, $V\in p^{-1}(M)$ is called \emph{the set of limit subalgebras of $\mathcal{F}$ at $m$}.
\end{definition}

 \subsection{The Helffer-Nourrigat cone of a singular foliation}\label{sec:Helffer}
The Helffer–Nourrigat cone, named after Bernard Helffer and Jean Nourrigat in \cite{androulidakis2022pseudodifferential} by Androulidakis, Mohsen, and Yuncken, becomes a fundamental object in noncommutative geometry. It was first implicitly introduced in \cite{Helffer1980}. The latter can be described easily in terms of an anchored bundle $A\to TM$ over the singular foliation $\mathcal F$. {Although this presentation differs formally from that of \cite{androulidakis2022pseudodifferential, debord:hal-05035630}, the underlying geometric object is the same: both descriptions encode the closure of regular annihilator spaces associated with the singular foliation.} 
This description is very similar to the construction of the Nash blow-up \cite{mohsen2021blow,louis2024nash,louis2023series}. We will see that the Helffer-Nourrigat cone is related to the Nash blow-up of the singular foliation.

\begin{definition}
Let $\rho\colon A\to TM$ be an anchored bundle over $ \mathcal F$.
Let $ \rho^* \colon  T^* M \to A^*$ be the dual of the anchor map. The \emph{Helffer-Nourrigat cone, associated with $ (A,\rho)$}, is the closed subset of $ A^*$ given by 
 \[{\mathrm{HN}}(\mathcal F):= \overline{ \cup_{m\in M_{\mathrm{reg}}} \rho_m^* (T_m^*M)}\subset A^*.\]We shall denote by $\mathfrak P\colon{\mathrm{HN}}(\mathcal F) \to M $ the restriction of the projection $A^* \to M$ to the ${\mathrm{HN}}(\mathcal F)$.
\end{definition}
Of course, the notation ${\mathrm{HN}}(\mathcal F)$ which does not mention the anchored bundle $A\to TM$ will be justified in Proposition \ref{prop:justification}.
\begin{remark}\label{rmk:cone}\begin{enumerate}
\item The closed subset ${\mathrm{HN}}(\mathcal F)\subset A^*$ is indeed a cone, as it is, by construction, stable by multiplication by a scalar.
    \item The Helffer-Nourrigat cone $\mathrm{HN}_{A,\rho}(\mathcal{F})\subset A^*$ can be seen as the union of all $r$-dimensional vector spaces in the closure of the image of the following well-defined section\[C \colon M_{\mathrm{reg}}\to \mathrm{Grass}_{r}(A^*), m\mapsto \mathrm{Im}(\rho^*_m)\]where $P\colon \mathrm{Grass}_{r}(A^*)\to M$ is the Grassmann bundle of $r$-dimensional subspaces in $A^*$. If we denote by $\mathfrak p\colon \overline{C(M_{\mathrm{reg}})}\to M$ the restriction of the projection $P$ to $\overline{C(M_{\mathrm{reg}})}$, we have that \[{\mathrm{HN}}(\mathcal F)=\cup_{m\in M}\mathfrak p^{-1}(m).\]None of the fibers of $\mathfrak p$ are empty because the projection $P$ has compact fibers. Moreover, for every $m\in M_{\mathrm{reg}}$ we have \[\mathfrak p^{-1}(m)=\left\{{\mathrm{Im}}(\rho^*_m)\right\}.\] 
\end{enumerate}
\end{remark}
The following proposition follows from item 2 of Remark \ref{rmk:cone}.
\begin{proposition}Let $\mathfrak P\colon {\mathrm{HN}}(\mathcal F) \to M $ be the Helffer-Nourrigat cone of a singular foliation associated with an anchored bundle $\rho\colon A\to TM$.
\begin{enumerate}
\item For every $m \in M$, the fiber of $\mathfrak P$ is non-empty, and is contained in the annihilator of the strong-kernel of $\rho$ at $m$.
\item  $\mathfrak P^{-1}(m)={\mathrm{Im}}(\rho^*_m)={\mathrm{ker}}(\rho_m)^{\circ}$, for every\footnote{We use the symbol $ {}^\circ$ for the annihilator.} $m \in M_{\mathrm{reg}}$.
\item For every $m \in M$, the fiber of $\mathfrak P$ is a union of sub-vector spaces of $ A_m^*$, all of dimension $r$.
\end{enumerate}
\end{proposition}

It is also interesting to see the Helffer-Nourrigat cone leaf-wise, that is, as a subset of $\coprod_{L\in M/\mathcal F}A_L^*$, where $A_L\to M$ is the holonomy Lie algebroid of $\mathcal{F}$ along a leaf $L\in M/\mathcal{F}$.

\begin{proposition}\label{prop:justification}Let $\mathfrak P\colon {\mathrm{HN}}(\mathcal F) \to M $ be the Helffer-Nourrigat cone of a singular foliation associated with an anchored bundle $\rho\colon A\to TM$. For every $m \in L\in M/\mathcal{F}$\begin{enumerate}
    \item The fiber of $\mathfrak P$ over $m$ is canonically included in $\left.A_L^* \right|_{m}$.
    \item The image of the inclusion $\mathfrak P^{-1}(m)\hookrightarrow A_L^*|_m$ is made of the union, of the annihilator of all limit subalgebras in $\mathfrak g_m(\mathcal{F})$, that is \[\mathrm{HN}(\mathcal{F})\simeq\cup_{V\in M_\mathcal{F}}\overline {V}^\circ\]where $\overline V\subset A_L|_m=\frac{A_m}{\mathrm{Sker \rho}_m}$ is the class of $V\subset \ker\rho_m$ in $\mathfrak g_m(\mathcal{F})$.
\end{enumerate}
In particular, $\mathrm{HN}(\mathcal{F})$ does not depend on the anchored bundle chosen in the construction.  
\end{proposition}

\begin{proof}
 Since the fibers of the holonomy algebroid $A_L$ at $m\in L$ is $\frac{A_m}{\mathrm{Sker}\rho_m}$, therefore \[A^*_L|_m\simeq \mathrm{Sker}({\rho|_L})_m^\circ=\mathrm{ker}\rho_m^\circ\subset A^*_m.\]
Since $\mathrm{HN}(\mathcal{F})$ is the closure of the image of the section \[C \colon M_{\mathrm{reg}}\to \mathrm{Grass}_{r}(A^*), m\mapsto \ker\rho_m^\circ,\] and that \begin{equation}\label{eq:duality-grass}
    \displaystyle \mathrm{Grass}_{-r}(A)\leftrightarrow \mathrm{Grass}_r(A^{*}),\; V\mapsto V^\circ
\end{equation}
is a canonical isomorphism, therefore, the Helffer-Nourrigat cone is exactly \[\mathrm{HN}(\mathcal{F})=\cup_{V\in M_\mathcal{F}} V^\circ.\]
Since every element $V\in M_\mathcal{F}$ is canonically included in $\mathfrak g_m(\mathcal{F})=\frac{\ker \rho_m}{\mathrm{Sker}\rho_m}$ (which only depends on $\mathcal{F}$)  (see \cite{louis2023series}), thus, the inclusion $V^\circ \subset A^*_L|_m$ is canonical.  
In addition, the fiber $\mathfrak P^{-1}(m)$ of the Helffer-Nourrigat cone over a point $m$ is the union of all annihilators $V^\circ$ of the points $V\in p^{-1}(m)$ of the fiber of the Nash resolution $p\colon M_\mathcal{F}\to M$ over $m$. This proves Item 1 and 2.
\end{proof}

\begin{remark}
Notice that through the isomorphism in Equation \eqref{eq:duality-grass} we have $M_\mathcal{F}\simeq \mathrm{HN}(\mathcal{F})$.
\end{remark}

Now, here is the main statement of this section.
\begin{proposition}Let $\mathfrak P\colon {\mathrm{HN}}(\mathcal F) \to M $ be the Helffer-Nourrigat cone of a singular foliation associated with an anchored bundle $\rho\colon A\to TM$.
\begin{enumerate}
\item If $(A,\rho)$ admits a Lie algebroid structure $[\,\cdot,\cdot]_A\colon \Gamma(A)\times \Gamma(A)\to \Gamma (A)$ with anchor map $\rho$, then the Helffer-Nourrigat cone ${\mathrm{HN}}(\mathcal{F})$ is a union of symplectic leaves of the linear Poisson structure on $A^*$ associated with the Lie algebroid bracket.
\item For every leaf $L\in M/\mathcal{F}$ the restriction $\mathrm{HN}(\mathcal F)|_L$ of the Helffer-Nourrigat cone to $L$ is a union of symplectic leaves of the linear Poisson structure on the dual $A_L^*$ of the holonomy Lie algebroid  $A_L\to L$.
\end{enumerate}
\end{proposition}

\begin{proof}
Assume that $A$ is a Lie algebroid. Recall that the canonical linear Poisson structure on $\mathrm{pr}\colon A^*\to M$ is given as follows \cite{Mackenzie, CFM}: To every section $a\in \Gamma(A)$ is associated a fiberwise linear function on the total space of the dual vector bundle $\mathrm{pr}\colon A^*\to M$ given by the evaluation map \(\mathrm{ev}_a\colon A^*\to \mathbb R, \alpha\mapsto \langle \alpha, a\circ \mathrm{pr}\rangle\). The fiberwise linear Poisson bracket \(\{\cdot\,,\cdot\}_{A^*}\) on $A^*$ is uniquely determined by the relation \begin{equation}\label{eq:Linear_Poisson}
    \{\mathrm{ev}_a, \mathrm{ev}_b\}_{A^*}=\mathrm{ev}_{[a,b]_A}, \; a, b\in \Gamma(A).
\end{equation} The relation \eqref{eq:Linear_Poisson} imposes that \[\{\mathrm{ev}_a, \mathrm{pr}^*f\}_{A^*}=\mathrm{pr}^*(\rho(a)[f])\;\text{and}\;
\;\{\mathrm{pr}^*f, \mathrm{pr}^*g\}_{A^*}=0,\]for all $a\in \Gamma(A)$ and $f, g\in \mathcal{O}_M$.

For every $a\in \Gamma(A)$, the Hamiltonian vector field $\mathcal{H}_a\in \mathfrak X(A^*)$ associated with the evaluation function $\mathrm{ev}_a\colon A^*\to \mathbb R$  is a linear vector field on $A^*$ which projects to $\rho(a)$, that is, the following diagram is a vector bundle morphism $$\xymatrix{A^*\ar[r]^{\mathcal{H}_a} \ar[d]_{\mathrm{pr}}&TA^*\ar[d]^{T\mathrm{pr}}\\M\ar[r]^{\rho(a)}&TM}$$Hence, the flow $\phi^{\mathcal{H}_a}_t\colon A^*\to A^*$ of the Hamiltonian vector field $\mathcal{H}_a$ is a isomorphism of vector bundles.
Now, the proof of Item 1 is done in two steps.
\begin{enumerate}
    \item [(a)] For $a\in \Gamma(A)$ with $X=\rho(a)$, the flow $\phi^{\mathcal{H}_a}_t\colon A^*\to A^*$ of the Hamiltonian vector field $\mathcal{H}_a$ on $A^*$ exists whenever the flow $\phi^X_t$ exists and the following diagram commutes \begin{equation}\label{eq:diag1}
\xymatrix{T^*M\ar[r]^{(T\phi_t^X)^*}\ar[d]_{\rho^*}&T^*M\ar[d]^{\rho^*}\\ A^*\ar[d]\ar[r]^{\phi_t^{\mathcal{H}_a}}&A^*\ar[d]^{}\\M\ar[r]^{\phi_t^X} &M.}
\end{equation}
 To see that, let $\mathcal{H}_X$ denotes the Hamiltonian vector field associated with $X\in \mathfrak X(M)$ w.r.t the canonical Poisson structure of $T^*M$. The vector field $(\mathcal{H}_X+\mathcal{H}_a, X)$ is a linear vector field on $T^*M\oplus A^*$ that is tangent to the graph of $\rho^*\colon T^*M\to A^*$$$ \left\{ (\xi, \rho^*(\xi) ) \,|\,  \xi \in T^*M \right\} \subset T^*M \oplus A^*. $$  Hence, the diagram \eqref{eq:diag1} commutes. The complete argument of the proof is very similar to \cite[Proposition 1.8]{louis2024nash}.  

\item [(b)] For every $m\in M$, the flow $\phi_t^{\mathcal{H}_a}$ (whenever it makes sense) sends $\mathrm{Im}{\rho^*_m}$ to  $\mathrm{Im}\rho^*_{\phi^X_t(m)}$, that is
$$\phi_t^{\mathcal{H}_a}|_{m}\left(\mathrm{Im}{\rho^*_m}\right)=\mathrm{Im}\rho^*_{\phi^X_t(m)}$$in view of Diagram \eqref{eq:diag1}. 
Since $\phi^X_t(M_{\mathrm{reg}})\subseteq M_{\mathrm{reg}}$, this implies that $  \phi^{\mathcal{H}_a}_t$ preserves the closure $\overline{ \cup_{m\in M_{\mathrm{reg}}} \mathrm{Im}\rho_m^*}\subset A^*$. 
Therefore, for every $a\in \Gamma(A)$, the flow $\phi_t^{\mathcal{H}_a}$ of $\mathcal{H}_a$, whenever it is defined, preserves the Helffer-Nourrigat cone $\mathrm{HN}(\mathcal{F})$, i.e.,  \begin{equation}
    \xymatrix{\mathrm{HN}(\mathcal{F})\ar[r]^{\phi^{\mathcal{H}_a}_t}\ar[d]_{\mathfrak P}&\mathrm{HN}(\mathcal{F})\ar[d]^{\mathfrak P}\\M\ar[r]^{\phi_t^X}&M}
\end{equation}
This means that any symplectic leaf of the canonical Poisson structure on $A^*$ 
that contains a point of $\mathrm{HN}(\mathcal{F})$ is entirely included in $\mathrm{HN}(\mathcal{F})$.  Item 2 is proved in the same manner as Item 1, since we may assume that $A_L=A$,
by extending $A_L$ to an anchored bundle over $\mathcal{F}$. This completes the proof.
\end{enumerate}
 \end{proof}

\subsection{Helffer-Nourrigat cone and the Nash algebroid }
In the following proposition, we relate the Helffer-Nourrigat cone with the Nash algebroid. 

\begin{proposition}\label{prop:NHandNash}
 Let $\mathcal{F}$ be a singular foliation and $\rho\colon A\to TM$ be an anchored bundle over $\mathcal F$. Let $(M_\mathcal{F}, \mathfrak D_\mathcal{F})$ be the Nash algebroid associated with $\mathcal
 F$.   Then, there exists a fiberwise injective\footnote{{But of course not injective.}} vector bundle morphism:
 $$  \xymatrix{\mathfrak j\colon \mathfrak D_\mathcal{F}^* \ar[d] \ar[r] & \ar[d] A^*\\
M_\mathcal{F} \ar[r]^{p}&M}.$$
The image $\mathfrak j(\mathfrak D_\mathcal{F}^*)\subset A^*$ coincides with the Helffer-Nourrigat cone ${\mathrm{HN}}(\mathcal F)$.
Moreover, \begin{enumerate}
    \item for every leaf $L\in M/ \mathcal F$, the image of $ \mathfrak j\colon \mathfrak D_\mathcal{F}^*\to A^*$ takes values in the annihilator of the strong-kernel of $\rho\colon A\to TM$,
which can be identified with $ A_L^*$.
\item the vector bundle morphism $\mathfrak j$,  induces a fiberwise injective vector bundle morphism\footnote{Here, we consider $\coprod_{L \in {\mathrm{Leaves}}} A_L^*$ as a “singular” vector bundle over $ M$: the ranks of the fibers vary.}
 $$  \xymatrix{\mathfrak J \colon {\mathfrak D_\mathcal{F}^*} \ar[d] \ar[r] & \ar[d] \coprod_{L \in M/\mathcal{F}} A_L^*\\
M_\mathcal F \ar[r]^{p}&M}$$
whose image is identified with the Helffer-Nourrigat cone  $ {\mathrm{HN}}(\mathcal F)$.
\end{enumerate}
\end{proposition}
\begin{proof}
We start by recalling a general construction. Given a vector bundle $ E \to M$ and a map $p\colon N \to M$, consider the pull-back bundle $p^*E \to N$. Now, let $T \subset  p^*E $ be a subvector bundle over $N$ then there is a natural fiberwise injective vector morphism 
 $$ \xymatrix{\left(\frac{p^*E}{T}\right)^*  \ar[d] \ar^{\mathfrak j}[r] & E^*\ar[d]\\ N\ar[r]&M}.$$
It consists in identifying for every point $n \in N$ the dual of the quotient space $\left(\frac{ p^*E_n}{T_n}\right)^*$ with the annihilator $ T^\circ_n \subset E^*_{p(n)}$ of $ T_n$, and to inject the later in $ E^*$.
Moreover, the image of $\mathfrak j $ can be described as follows. For every $m \in M$,  its intersection with $ E_m^*$ is a union of vector spaces of dimension $ {\mathrm{rk}}(E)- {\mathrm{rk}}(T)$. We therefore call \emph{image cone} this image.

We apply this construction to
\begin{enumerate}
\item $E:=A $, with $\rho\colon A\to TM$ an anchored bundle over $ \mathcal F$,
\item $ N:=M_\mathcal{F}$ the Nash blowup\footnote{Of course, it may not be a manifold, but this defect  has no practical consequence here.} of \S \ref{sec:Nash}, computed with respect to $\rho\colon A\to TM$, 
\item $T$ the tautological bundle on the Grassmannian  $ {\mathrm{Grass}}_{-r}(A)$, restricted to $M_\mathcal{F}$.
\end{enumerate}

In this case, the quotient $p^* E / T $ is the Nash algebroid $\mathfrak D_\mathcal{F}$\footnote{That we recall to be the Lie algebroid of Mohsen's groupoid \cite{mohsen2021blow}.} associated to the Nash blow up of $ \mathcal F$.
We therefore obtain a fiberwise injective vector bundle morphism:
 $$  \xymatrix{\mathfrak j\colon \mathfrak D_\mathcal{F}^* \ar[d] \ar[r] & \ar[d] A^*\\
M_\mathcal{F} \ar[r]^{p}&M}.$$

Since $p\colon M_\mathcal{F}\to M$ is proper and the image of $ \mathfrak j$ over $ \pi^{-1}(M_{\mathrm{reg}})$ (i.e., the regular part, which is dense) coincides with the image of $ \rho^*$ (i.e., the annihilator of the kernel of $ \rho$), then the image cone $\mathfrak j(\mathfrak D_\mathcal{F}^*)$ in this case coincides with the Helffer-Nourrigat cone. Moreover, for every leaf $L$ of $ \mathcal F$, the image of $ \mathfrak j$ takes values in the annihilator of the strong-kernel of $\rho\colon A\to TM$,
which can be identified with $ A_L^*$ (the dual of the holonomy Lie algebroid).
The previous vector bundle morphism, therefore,  induces a vector bundle morphism  (which is still fiberwise injective)
 $$  \xymatrix{\mathfrak J \colon {\mathfrak D_\mathcal{F}^*} \ar[d] \ar[r] & \ar[d] \coprod_{L \in M/\mathcal{F}} A_L^*\\
M_\mathcal F \ar[r]^{p}&M}$$
whose image, is again, the Helffer-Nourrigat  cone  $ {\mathrm{HN}}(\mathcal F)$.
\end{proof}

\begin{example}
Let $\mathcal{F}$ be a Debord singular foliation, i.e., there exists a vector bundle $A\longrightarrow M$ such that $\Gamma(A)\simeq \mathcal{F}$. This isomorphism is given by a vector bundle morphism, $A\stackrel{\rho}{\rightarrow}TM$ which is injective on the open dense subset $M_{\mathrm{reg}}$. Here $M_\mathcal{F}\simeq M$ and the Helffer-Nourrigat cone is $\cup_{m\in M}A^*_m=A^*$.\end{example}

\begin{example}
Let $\mathcal F$ be the singular foliation on $\mathbb R^3$ generated by the image of the cotangent Lie algebroid $\pi^{\#}\colon T^*\mathbb R^3\to T\mathbb R^3$,  $\pi^\sharp(dx)= z\frac{\partial}{\partial y}-y\frac{\partial}{\partial z}$,\, $\pi^\sharp(dy)= x\frac{\partial}{\partial z}-z\frac{\partial}{\partial x}$,\, $\pi^\sharp(dz)= y\frac{\partial}{\partial x}-x\frac{\partial}{\partial y}$. Here, $M_\mathrm{reg}$ is $\mathbb R^{3}\setminus \{0\}$. The Nash blowup space  $ M_\mathcal{F}$ is the usual blowup $\mathrm{Bl}_0(\mathbb R^3)$ of $\mathbb R^3$ at the origin. The Helffer-Nourrigat cone  $\mathfrak P\colon{\mathrm{HN}}(\mathcal F) \to \mathbb R^3$ is

\[\begin{cases}
    \mathfrak P^{-1}(v)=\langle v\rangle^\perp\simeq T_{\frac{v}{||v||}}^*S^2, & v\in \mathbb R^{3}\setminus\{0\}\\\mathfrak{P}^{-1}(0)= \bigcup_{\ell\in \mathbb P^2}\ell^\perp= \mathbb R^3.
\end{cases}\]
\end{example}

I acknowledge discussions with Camille Laurent-Gengoux; Fani Petalidou; Mohsen Masmoudi and Cédric Rigaud for the following example.

\begin{example}
    Let $\mathcal{F}$ be the singular foliation on $ M=\mathbb R^d$ made of all vector fields vanishing at the origin.  The singular foliation $\mathcal{F}$ is given by the action Lie algebroid $A=\mathfrak {gl}_d(\mathbb R)\ltimes \mathbb R^d\to \mathbb R^d$. The Nash blow-up $p\colon M_\mathcal{F}\to M$ coincides to the usual blow-up of $\mathbb R^d$  at the origin. The Helffer-Nourrigat cone of $\mathcal{F}$ is given as follows:
\begin{enumerate}
    \item if $m$ is not the origin, the fiber of the Helffer-Nourrigat cone over $m$ is isomorphic to $\mathbb R^d $, 
    \item if $m$ is the origin, the fiber is made of square $d \times d $ matrices of rank $\leq 1$.
\end{enumerate}
\end{example}

\begin{example}
Consider the singular foliation $\mathcal{F}\subset \mathfrak X(\mathbb R^2)$ generated by the vector fields \[x^2\frac{\partial }{\partial x},\, y^2\frac{\partial }{\partial x},\, xy\frac{\partial }{\partial x},\, x^2\frac{\partial }{\partial y},\, y^2\frac{\partial }{\partial y},\, xy\frac{\partial }{\partial y}.\]
$\mathcal{F}$ is the image of the anchored bundle $A=\mathbb R^2\times \mathbb R^{6}$. The Helffer-Nourrigat cone  $\mathfrak P\colon{\mathrm{HN}}(\mathcal F) \to \mathbb R^2$ is equal to \begin{enumerate}
    \item $T^*_{(x,y)}\mathbb R^2$ at $(x,y)\in \mathbb R^2\setminus \{0\}$.
    \item  a line in the cone $uv=w^2$ and two points on that line, at the origin.
\end{enumerate}
\end{example}

\section{Longitudinal differential operators}\label{sec:LDO}
\subsection{The  universal enveloping
algebra of a Lie-Rinehart algebra}
Let us first recall the definition of a Lie-Rinehart algebra on a commutative associative unital algebra $\mathcal{O}$.
\begin{definition}\cite{MR2075590,MR1625610, CamilleLouis}
A \emph{Lie-Rinehart algebra} over $\mathcal O$ is a triple $(\mathcal A , [\cdot, \cdot]_\mathcal A , \rho_{\mathcal A})$ with $ \mathcal A$ an $ \mathcal O$-module, $[\cdot, \cdot]_\mathcal A  $ a Lie algebra bracket on  $\mathcal A $, and $ \rho_\mathcal A \colon \mathcal A \longrightarrow  {\mathrm{Der}}(\mathcal O)$ an $ \mathcal O$-linear Lie algebra morphism called \emph{anchor map}, satisfying the so-called \emph{Leibniz identity}:
 $$   [  a,  f b ]_\mathcal A  = \rho_\mathcal A (a ) [f] \, b + f [a,b]_\mathcal A  \hbox{ for all $ a,b \in \mathcal A, f \in \mathcal O$}.$$\end{definition}

Let $(\mathcal{A}, [\cdot,\,\cdot]_\mathcal{A}, \rho_\mathcal{A})$ be a Lie-Rinehart algebra over $\mathcal{O}$. To define the universal enveloping algebra of $\mathcal{A}$ (see \cite{rinehart1963differential, zbMATH01287585, moerdijk2008universalenvelopingalgebralierinehart,zbMATH07453484} or \cite{maakestad2022envelopingalgebraliealgebra}), proceed as follows:
\begin{enumerate}
\item 
consider the universal enveloping algebra $U(\mathcal F)$ of the \emph{Lie} $\mathbb R$-algebra $\mathcal A $ to which we give a $\mathcal{O}$-module structure as \[f(a_1
\cdots\cdot \cdots a_k):=(fa_1)\cdots \cdot\cdots  a_k\]
for $a_1,\ldots, a_k\in \mathcal A $ and $f \in \mathcal O$.
\item divide $ U(\mathcal A)$ by the ideal $I$ generated  by: 
  \begin{equation}\label{eq:defideals}   a \cdot (f b) - (fa) \cdot b - \rho(a)[f] b \end{equation}
  where $a,b \in \mathcal A $ and $f \in \mathcal O$. 
  \end{enumerate}
  \begin{definition}
      The quotient $\mathcal U(\mathcal A):=U(\mathcal{A})/I$ is called the \emph{universal enveloping algebra of the Lie-Rinehart algebra  $\mathcal A $}.
  \end{definition}

  \begin{remark}
      The universal enveloping algebra $\mathcal U(\mathcal A)$ of a Lie-Rinehart algebra $\mathcal A $ is a filtered algebra $\mathcal U^{\leq \bullet}(\mathcal{A}) $, with $\mathcal U^{\leq k}(\mathcal{A})$ being the subspace generated by monomials of degree $ \leq k$.
  \end{remark}

 \begin{definition}
Let $(\mathcal{A}, [\cdot,\,\cdot]_\mathcal{A}, \rho_\mathcal{A})$ be a Lie-Rinehart algebra.\begin{enumerate}
    \item If $\mathcal{A}=\mathcal{F}\subseteq \mathfrak X(M)$ is a singular foliation on $M$, we say that $\mathcal U(\mathcal F)$ is the \emph{universal enveloping algebra of the singular
foliation $\mathcal{F}$}\item  If $(\mathcal{A}, [\cdot,\,\cdot]_\mathcal{A}, \rho_\mathcal{A})$ is a Lie algebroid over $M$ (i.e., $\mathcal{A}=\Gamma(A)$ for some vector bundle $A\to M$), we say that $\mathcal U( A)$ is the \emph{universal enveloping algebra of the Lie algebroid $A$} \cite{Moerdijk2010}.
\end{enumerate}
 \end{definition}

\subsection{Differential operators on a singular foliation}
\begin{definition}
    Let  $(M,\mathcal F) $ be a singular foliation, we call \emph{longitudinal differential operator} any linear combination of operators of the form
\begin{equation}\label{eq:monomials} 
\begin{array}{rcl} \mathcal C^\infty(M)&\to & \mathcal C^\infty(M)\\f &\mapsto &X_1 \circ \dots \circ X_k [f] \end{array}\end{equation}
with $ X_1, \dots, X_k \in \mathcal F$. 
We denote by $ {\mathrm{Diff}}(\mathcal F)$ the algebra of longitudinal differential operators.
\end{definition}


Now, consider the natural algebra morphism from  $U(\mathcal F) $ to $ {\mathrm{Diff}}(\mathcal F) $ defined by 
\begin{equation}\label{eq:realization} \begin{array}{rcl} U(\mathcal F) & \to & {\mathrm{Diff}}(\mathcal F)\\ X_1 \cdot \dots \cdot X_k &\mapsto& 
X_1 \circ \dots \circ X_k \end{array}.
 \end{equation}
 The latter goes to the quotient, on the universal enveloping algebra of $\mathcal F$ to define a surjective algebra morphism
\begin{align}\label{eq:realization10}
  \mathcal{U}(\mathcal{F})&\to\mathrm{Diff}(\mathcal{F})\\ \nonumber P&\mapsto \underline{P} 
\end{align}
that we call \emph{realization of $ \mathcal{U}(\mathcal F)$.} 
  
\begin{remark}The realization map $ \mathcal{U}(\mathcal F)\to \mathrm{Diff}(\mathcal{F})$ is not injective in general. For instance, the singular foliation on $\mathbb R^4$ generated by all vector fields vanishing at zero is a counter-example. Indeed, if $x_1,x_2, x_3$ and $x_4$ are the global coordinates of $\mathbb R^4$, then  $\mathcal{F}$ is generated by the vector fields $x_i\frac{\partial}{\partial x_j}$ for $i,j=1,\ldots,4$. For the non-zero element $P=x_1\frac{\partial}{\partial x_2}\cdot x_3\frac{\partial}{\partial x_4}-x_1\frac{\partial}{\partial x_4}\cdot x_3\frac{\partial}{\partial x_2}\in \mathcal{U}(\mathcal{F})$, we have $\underline{P}=x_1\frac{\partial}{\partial x_2 \partial x_4}-x_1\frac{\partial}{\partial x_4\partial x_2}=0$.
\end{remark}
However, 

\begin{proposition}
    If $ \mathcal F$ is a Debord singular foliation, (that is, $\mathcal{F}$ is a locally free $C^\infty(M)$-module) then The realization map $ \mathcal{U}(\mathcal F)\to \mathrm{Diff}(\mathcal{F})$ is a bijection. 
\end{proposition}

\begin{definition}
Let $\mathcal{F}$ be a singular foliation. Elements in  $ {\mathrm{Diff}}(\mathcal F) $ 
 of the form \eqref{eq:monomials} shall be called \emph{monomials of degree $k$}, and we say that a longitudinal differential operator is of \emph{degree $\leq k $} if it is a sum of monomials of degree $\leq k$.
\end{definition}

\begin{remark}
    The degree defines an increasing filtration on the algebra $ {\mathrm{Diff}}(\mathcal F) $, making it a filtered algebra $  \left(\mathrm{Diff}^{\leq k}(\mathcal F) \right)_{k \geq 0}$. Similarly, the algebra $\mathcal U(\mathcal F) $ comes equipped with a filtration $  \left(\mathcal U^{\leq k} (\mathcal F)\right)_{k \geq 0}$ defined in a same manner. The realization \eqref{eq:realization10}  is a morphism of filtered algebras $\mathcal{U}(\mathcal{F})\to\mathrm{Diff}(\mathcal{F})$.
\end{remark}

\subsection{What is the symbol of a longitudinal differential operator?}
Let $\mathcal{F}$ be a singular foliation. To start with, we shall define two ``symbols'':
\begin{enumerate}
\item the symbol of an element in the universal enveloping algebra $ \mathcal U(\mathcal F)$ of $ \mathcal F$, 
\item the symbol  of a longitudinal differential operator, i.e., an element in ${\mathrm{Diff}}(\mathcal F) $. 
\end{enumerate}

Before defining these two symbols,
let us very briefly recall the definition of the symbol in the context of Lie algebroids as in \cite{NistorPingAlan,zbMATH07453484}. Let $(A \to M, [\cdot\,,\cdot]_A, \rho)$ be a Lie algebroid over $M$. 
\begin{enumerate}
\item Upon choosing a Lie algebroid connection, there is a grading preserving coalgebra isomorphism $\mathcal U(A) \simeq \Gamma(S(A)) $ from the universal enveloping algebra of $A$ to the symmetric algebra of $B$. See, e.g., \cite{zbMATH07369649}.
\item Since this isomorphism preserves the grading, it induces a $\mathcal C^\infty(M) $-linear map\footnote{We insist that it is $ S^k(A)$, not $ S^{\leq k}(A)$: we project on the top component, i.e., the space generated by monomials of degree $k $.}   \[\mathcal U^{\leq k}(A)\to \Gamma(S^k (A)).\] 
\item This map does not depend on the choice of a Lie algebroid connection\footnote{Only the top component does not depend on the connection, the low components do. This independence means that we could work with local connections, so that what we say here extends to the real analytic or holomorphic settings.}.
\end{enumerate}

\begin{remark}
    For any vector bundle $A$ over $M$ and $k\in \mathbb N$, a section $\xi$ of $S^k (A)\to M$ may be interpreted as a function $f_\xi$ on $A^*$ as\[ f_\xi(\alpha)=\langle \alpha^{\odot k}, \xi\rangle,\;\text{for all}\,\; \alpha\in \Gamma(A^*)\] which is fiberwise polynomial and homogeneous of degree $k$. 
\end{remark}

\begin{definition}
Let $(A, [\cdot\,,\cdot]_A, \rho)$ a Lie algebroid over $M$. For every $P \in \mathcal U^{\leq k}(A)\simeq \Gamma(S^{\leq k}(A))$,  the \emph{symbol} $\sigma_P $ of $P$ is the associated function on $A^*$. 
\end{definition}

\subsubsection{Definition I: through the holonomy Lie algebroids}\label{def:I}Let $\mathcal{F}$ be a singular foliation. We now define the symbol as an element in $ \mathcal U (\mathcal F)$. For every leaf $ L$ of a singular foliation $\mathcal F $ on a manifold $M$, there exists a natural restriction map
 $$ \mathcal F \longrightarrow \Gamma(A_L), X\longmapsto X|_L $$
where $A_L $ is the holonomy Lie algebroid of the leaf $L$. This map is $\mathcal C^\infty(M) $-linear and is a Lie algebra morphism, so that it induces an algebra morphism 
 $$\begin{array}{ccc} U (\mathcal F) & \to &  U(\Gamma(A_L)) \\ P& \mapsto &P|_L \end{array} $$ from the algebra of longitudinal operators $U (\mathcal F)  $ to the universal enveloping algebra $U(\Gamma(A_L)) $.
 It goes to the quotient to induce an algebra morphism 
 \begin{equation} \label{eq:PofL}\begin{array}{ccc} \mathcal U (\mathcal F) & \to &  \mathcal U(A_L) \\ P& \mapsto &P|_L \end{array} \end{equation}  from the universal enveloping algebra $\mathcal U (\mathcal F)  $ of the singular foliation $ \mathcal F$ to the universal enveloping algebra $\mathcal U(A_L) $ of the Lie algebroid $A_L\to L$.

Therefore, for a given $P \in \mathcal U^{\leq k}(\mathcal F)$ of degree  $\leq k $, consider, for every leaf $L$, the symbol $\sigma_{P|_L} $. It is a fiberwise homogeneous of degree $k$ smooth function on $ A_L^*$.

 \begin{remark}
     For any leaf $L$, the realization of $P$, restricted to $L$,  is a differential operator in the usual sense on $A_L$. For a regular leaf, $ A_L=TL$ and $\sigma_{P|_L}$ is simply the usual symbol \cite{NistorPingAlan} of this differential operator.
 \end{remark}
 
 \begin{definition}
     We call \emph{symbol of $P\in \mathcal{U}^{\leq k}(\mathcal{F})$},  collection $(L, \sigma_{P|_L)})_{L\in M/\mathcal{F}}$  made of the functions $ \sigma_{P|_L}$. We denote it by $\sigma_P $.
 \end{definition}

\subsubsection{Definition II: through Nash blowup}\label{def:II}
Let $(M, \mathcal{F})$ be a singular foliation. We shall assume that
\begin{enumerate}
    \item $\mathcal{F}$ is finitely generated so that there exist anchored bundles $A\to TM$ whose images are $\mathcal{F}$.

\item that the regular leaves of $(M,\mathcal F) $ are all the same dimension, so that the Nash blowup
 $$ \left(M_\mathcal{F}, \mathfrak D_\mathcal F \right)   $$ makes sense. See Section \ref{sec:kernel}.
\end{enumerate} Recall from \S \ref{sec:Nash} that the Nash blowup is a Debord singular foliation, whose associated Lie algebroid is the canonical quotient bundle $\mathfrak{D}_\mathcal{F}$ of the Grassmann bundle, restricted to $ M_\mathcal F$ \cite{louis2024nash,louis2023series}. 
Consider the lift
 $$ \mathcal F \to \Gamma(\mathfrak{D}_\mathcal{F})$$
 mapping $X \in \mathcal F$ to the unique section of $\Gamma(\mathfrak{D}_\mathcal{F})$ corresponding to $ p^!X$.
 This lift is \begin{enumerate}
     \item  a $\mathcal C^\infty(M) $-module morphism and \item  a Lie algebra morphism.
 \end{enumerate}
 This implies that it rises to an algebra morphism
   $$  \mathrm{Diff} (\mathcal F) \longrightarrow \mathcal U\left({\mathfrak{D}_\mathcal{F}}\right),\; D \longmapsto  p^!{D}$$
   where $\mathcal U\left({\mathfrak{D}_\mathcal{F}}\right)$ stands for the universal algebra 
 of the Nash-blowup Lie algebroid $ {\mathfrak{D}_\mathcal{F}}$. Since $ p^!{D}$ is an element of degree $\leq k$ of the universal Lie algebroid $ {\mathfrak{D}_\mathcal{F}}$, its symbol is a (perfectly well-defined) element of $\Gamma(S^k  {(\mathfrak{D}_\mathcal{F})})$, or, equivalently, a function on $\mathfrak{D}_\mathcal{F}^*$, which is fiberwise polynomial and homogeneous of degree $k$.
 Since its restriction to the regular part is entirely determined by the image of $D$ through the realization map, and since this symbol continuously depends on the base point by construction, the next definition makes sense. 

 \begin{definition}
 Let $D$ be a longitudinal differential operator of degree $\leq k $ of $\mathcal{F}$.
     We call \emph{symbol} of $D$ of degree $k$ the symbol of the element $p^! {D}$ in the universal algebra $\mathcal U({\mathfrak{D}_\mathcal{F}}) $ of the Nash blowup algebroid ${\mathfrak{D}_\mathcal{F}}$.

     By construction, this symbol that we denote $\sigma_{D} $ is a fiberwise homogeneous of degree $k$ polynomial function on the dual of the Nash blowup Lie algebroid ${\mathfrak{D}_\mathcal{F}}$.
 \end{definition}

\subsection{Relation between Definition I and II}
We have defined in \S \ref{def:I} and \S \ref{def:II} two symbols associated with a longitudinal differential operator on a singular foliation $\mathcal
F$: \begin{enumerate}
    \item one for the universal enveloping algebra $\mathcal U(\mathcal F) $, defined as a family functions on the duals of the holonomy algebroids $(A_L)_{L\in M/\mathcal F}$ of $\mathcal F $\item 
    and one for the algebra of longitudinal differential operators ${\mathrm{Diff}}(\mathcal F)$, defined as a function on the dual of the Nash algebroid $(M_\mathcal F, \mathfrak D_\mathcal F)$.
\end{enumerate}
Recall that by Proposition \ref{prop:NHandNash}, the Nash algebroid and the holonomy algebroids are related through a fiberwise injective vector bundle morphism $\mathfrak J\colon \mathfrak D_\mathcal{F}^*\to \coprod _{L\in M/\mathcal{F}}A_L^*$ over $p\colon M_\mathcal F\to M$.

These two symbols are related as follows.

\begin{proposition}
\label{prop:pullbackHN}
Let $D \in {\mathrm{Diff}}^{\leq k}(\mathcal F)$  be a longitudinal differential operator of degree $\leq k $ and let $P \in \mathcal U^{\leq k}(\mathcal F)$ be any element whose realization is $D$. The  symbol $ \sigma_D$ is the pull-back through  $\mathfrak J$ $$  \xymatrix{\mathfrak J \colon {\mathfrak D_\mathcal{F}^*} \ar[d] \ar[r] & \ar[d] \coprod_{L \in M/\mathcal{F}} A_L^*\\
M_\mathcal F \ar[r]^{p}&M}$$ of the restriction of the  symbol $\sigma_P$ of $P$ to the Helffer-Nourrigat cone bundle of $ \mathcal F$.  Namely,
 $$  \sigma_D = \mathfrak J^* \left. \sigma_P \right|_{{\mathrm{HN}}(\mathcal F)}  .$$
\end{proposition}
\begin{proof}
It holds true on the regular part $M_{\mathrm{reg}}$ which is an open dense subset of $M$. The result follows by density.
\end{proof}

\begin{remark}
A notable consequence of this statement is that the restriction to ${\mathrm{HN}}(\mathcal F)$ of the symbol of an element $P\in \mathcal{U}(\mathcal{F})$ depends only on its realization $ \underline{P}$. 

{In particular, Proposition \ref{prop:pullbackHN} shows that the Helffer–Nourrigat cone is precisely the geometric locus on which the symbol of a longitudinal differential operator becomes intrinsic, independently of the choice of representative in $\mathcal U(\mathcal F)$.}
\end{remark}

\subsection{Longitudinally elliptic differential operator}
Now, recall that a differential operator is said to be elliptic if its symbol vanishes only at the origin.  For a Lie algebroid \cite{NistorPingAlan}, an element in the universal Lie algebra is said to be elliptic if its symbol vanishes only along the zero section.

Here is, in our opinion, the correct definition for a singular foliation. 

\begin{definition}\label{def:symbol-fol}
A longitudinal differential operator $D$ of degree $k$  on a singular foliation $\mathcal{F}$  is said to be \emph{longitudinally elliptic} if its symbol\footnote{Recall that the latter is a function on the dual $\mathfrak D_\mathcal{ F}^*$ of the Nash blowup Lie algebroid $\mathfrak D_\mathcal{ F}\to M_\mathcal{F}$ of $\mathcal{F}$.} $ \sigma_{D}$ is a strictly positive function, except on the zero section.
\end{definition}

\begin{remark} Let $D$ be a longitudinal differential operator on  a singular foliation $\mathcal F$.
    \begin{enumerate}
        \item $D$ is longitudinally elliptic if and only if  $ p^!D$ is elliptic for the Nash blowup Lie algebroid $ {\mathfrak{D}_\mathcal F}\to M_\mathcal F$ in the sense of \cite{NistorPingAlan, zbMATH06351314}, which seems to us to be a convincing justification of the notion.
\item Now, choose some $ P \in \mathcal U(\mathcal F)$ whose realization is $D$.
The symbol $ \sigma_P$ does \emph{not} need to be strictly positive outside the Helffer-Nourrigat cone by Proposition \ref{prop:pullbackHN},  so our definition does \emph{not} imply that $ P|_L \in \mathcal U(A_L)$ (see Equation \eqref{eq:PofL}) is elliptic for all leaves (although it certainly has to be elliptic on regular leaves).
    \end{enumerate}
\end{remark}

{Definition \ref{def:symbol-fol} above should be understood as a geometric ellipticity condition rather than a hypoellipticity criterion in the sense of Helffer-Nourrigat. In general, maximal hypoellipticity for longitudinal operators on singular foliations depends on finer representation-theoretic properties of the symbol, as shown in \cite{androulidakis2022pseudodifferential}, whereas Definition \ref{def:symbol-fol} only requires positivity on the desingularized symbol space $\mathfrak D_\mathcal F^*$. Nevertheless, this condition is natural because, after pullback to the Nash algebroid, the operator becomes elliptic in the usual Lie algebroid sense \cite{NistorPingAlan}. This suggests that the Nash resolution isolates the geometric part of ellipticity underlying the analytic conditions appearing in the Helffer–Nourrigat theory. A more detailed investigation of this relationship will be carried out in forthcoming work.}\\

\noindent
\textbf{Acknowledgments.} I acknowledge valuable discussions with C. Laurent-Gengoux, and L. Ryvkin during the preparation of this paper.

\bibliographystyle{alpha}
\bibliography{main}
\vfill
\begin{center}
    \textsc{Department of Mathematics, University of Illinois Urbana-Champaign\\ 1409 W. Green Street, Urbana, IL 61801, USA}
\end{center}

\end{document}